\documentclass[11pt]{article}
\usepackage{amsmath, amssymb, amscd, amsthm, amsfonts,bm}
\usepackage{graphicx,cite,float,booktabs}
\usepackage{hyperref,subcaption,verbatim}
\usepackage[linesnumbered,ruled,vlined]{algorithm2e}
\usepackage{algorithmic}

\title{Gaussian-process surrogate indicators for residual-based adaptive GMsFEM}
\author{Siqing Liu\thanks{Department of Mathematics, The Chinese University of Hong Kong} 
\and Eric Chung\thanks{Department of Mathematics, The Chinese University of Hong Kong}  \and Yiran Wang\thanks{Department of Mathematics, The University of Alabama} }
\date{}

\newtheorem{theorem}{Theorem}

\newtheorem{proposition}[theorem]{Proposition}
\newtheorem{corollary}[theorem]{Corollary}
\theoremstyle{definition}
\newtheorem{assumption}[theorem]{Assumption}
\theoremstyle{remark}
\newtheorem{remark}{Remark}

\newcommand{\rr}{\mathbb{R}}
\newcommand{\mc}[1]{\mathcal{#1}}
\newcommand{\mr}[1]{\mathrm{#1}}
\newcommand{\mf}[1]{\mathbf{#1}}
\newcommand{\mb}[1]{\mathbb{#1}}

\newcommand{\omi}{\omega_i}
\newcommand{\ze}{\zeta}
\newcommand{\ka}{\kappa}

\newcommand{\lam}{\lambda}
\newcommand{\bbet}{\bm{\beta}}

\newcommand{\rkhs}{\mathcal{H}_c}
\newcommand{\ums}{u_{\mathrm{ms}}}

\newcommand{\hmns}{$H^{-1}$-residual adaptive GMsFEM}

\newcommand{\vms}{V_{\mathrm{ms}}}
\newcommand{\hmo}{H^{-1}}
\newcommand{\vsnap}{V_{\mathrm{snap}}}

\newcommand{\voff}{V_{\mathrm{off}}}

\newcommand{\gp}{\mathcal{GP}}

\newcommand{\cerr}{C_{\mathrm{err}}}

\newcommand{\emind}{\eta_{m,i}^2}

\newcommand{\snap}{\mathrm{snap}}

\newcommand{\oprt}[1]{-\nabla\cdot(\kappa(x)\nabla #1)}
\newcommand{\brcs}[1]{\left\{#1\right\}}
\newcommand{\norm}[1]{\Vert #1 \Vert}

\newcommand{\keywords}[1]{
    \vspace{1ex}
    \noindent \textbf{Keywords:} #1
}

\usepackage{xcolor}
\begin{document}

\maketitle

\begin{abstract}
Residual-based adaptive GMsFEM for high-contrast elliptic problems repeatedly evaluates local weighted $H^{-1}$ indicators on every coarse neighborhood, making indicator evaluation a recurring cost in repeated-query settings. We introduce a non-intrusive Gaussian-process (GP) surrogate for the indicator scores used in D\"orfler marking. The operational predictor is the GP posterior mean, algebraically equivalent to a kernel ridge regression (KRR) estimator under the stated convention; it uses compressed local solution and spectral features without changing the multiscale solve, local spectral construction, or basis enrichment. A nonuniform perturbed-marking result quantifies how pointwise score errors affect the exact indicator mass captured by surrogate-selected neighborhoods, while a conditional bounded-discrepancy KRR pathway identifies sufficient assumptions for such score bounds. In controlled held-out in-distribution tests, the surrogate-guided method gives error-versus-DoF trends comparable with classical $H^{-1}$-residual offline adaptivity and evaluates the online indicator component $2.0$--$2.1$ times faster, excluding offline data generation and GP training.
\end{abstract}


\section{Introduction}\label{section-introduction}

Multiscale phenomena are widely present in nature, such as flow in porous media, biological transport, and atmospheric turbulence. Most problems are modeled by partial differential equations (PDEs) with high-contrast multiscale coefficients. Solving these equations directly using traditional numerical methods can be very time-consuming \cite{efendiev_multiscale_2004,allaire_multiscale_2005}. Therefore, some types of model reduction methods have been developed to solve these multiscale problems. The methods can be categorized into global \cite{hinze_proper_2005,ghommem_mode_2013} and local approaches \cite{hou_multiscale_1997,efendiev_generalized_2013}. Global model reduction methods require costly computations in the offline stage and lack local adaptivity. In this paper, our research focuses on the local multiscale model reduction methods that can add local basis functions adaptively.

The local model reduction methods include two classes: homogenization methods \cite{durlofsky_numerical_1991,chen_coupled_2003} and multiscale methods \cite{hou_multiscale_1997,efendiev_generalized_2013,kim_data-driven_2024,nikiforov_meshfree_2023}. Many of these methods solve multiscale PDEs on the coarse mesh determined by a fine grid. Of the existing multiscale methods,  the generalized multiscale finite element method (GMsFEM) is one of the most widely used due to its flexibility with regard to the number of local basis functions \cite{efendiev_generalized_2013}. The main idea is to divide the computation into the offline and online stages, construct the space of local basis functions in the offline stage, and use them to obtain multiscale basis functions in the online stage. However, the number of multiscale basis functions on each coarse region cannot be added adaptively based on local information in the online stage.

To tackle this problem, the adaptive GMsFEM has attracted significant attention \cite{chung2023multiscalebook}. A classical one is the \hmns \cite{chung_adaptive_2014}, which serves as the foundation for many subsequent approaches \cite{chung_goal-oriented_2016,wang_adaptive_2023,wang_ams-net_2022,xu_adaptive_2023}. This adaptive method is based on an error indicator defined as the $\hmo$ norm of the residual. The error indicator can determine the number of local basis functions to add in each iteration. It was shown that the weighted $\hmo$-norm residual enables the error indicator to be robust for high-contrast media \cite{chung_adaptive_2016}. However, the computation of this type of norm is affected by the local heterogeneities, which increase the computational burden. To address this issue, the goal-oriented adaptive GMsFEM was developed \cite{chung_goal-oriented_2016}, but this method needs to solve both primal and dual multiscale problems. Moreover, the residual-driven basis construction method without adaptivity is also studied in recent years \cite{wang_localglobal_2022}.

Machine-learning surrogates can reduce the cost of repeatedly evaluated numerical quantities without replacing the governing PDE solver \cite{chen_solving_2021,mora_operator_2025,cuomo_scientific_2022,brunton_promising_2024}. Here the surrogate target is deliberately narrow: it is the local indicator score used to rank neighborhoods for offline enrichment. The operational predictor is the GP posterior mean, algebraically equivalent to a KRR estimator under the convention stated in Section \ref{subsec:inference of error indicators} \cite{rasmussen_gaussian_2005,kanagawa_gaussian_2018}. The GP formulation supplies marginal-likelihood training and automatic relevance determination (ARD) kernel calibration, while providing a natural route to future uncertainty-aware marking. Posterior variance is not used by the present algorithm.

We therefore consider a fixed source term, discretization, and admissible coefficient distribution in which one fitted model is reused and its offline cost is amortized over many adaptive solves. The surrogate replaces repeated local dual-norm evaluations for marking, while the multiscale solution and selected offline bases remain standard GMsFEM components. Our contributions are: (i) a non-intrusive posterior-mean indicator surrogate based on compact local solution and spectral features, with all neighborhood scores assembled in one batchable matrix evaluation; (ii) an explicit treatment of deterministic labels, feature-compression discrepancy, offline cost, and the repeated-query applicability boundary; (iii) a nonuniform perturbation result linking pointwise surrogate-score accuracy to the exact indicator mass captured by D\"orfler marking, together with a conditional RKHS/KRR pathway to such score bounds; and (iv) controlled held-out in-distribution validation showing comparable error-versus-DoF trends and a $2.0$--$2.1$ online indicator-evaluation speedup in the four reported configurations. Local feature construction, kernel evaluation, and offline label generation admit parallel work across samples or neighborhoods, although dense GP factorization remains centralized.

The remainder of the paper is organized as follows. Sections \ref{sec:GMsFEM} and \ref{sec:hminus1 adaptive GMsFEM} recall the GMsFEM construction and the classical $H^{-1}$-residual adaptive procedure. Section \ref{sec:GP models} introduces the GP surrogate indicators. Section \ref{sec:analysis} gives the conditional marking analysis. Numerical results are presented in Section \ref{sec:Numerical experiments}.

\section{GMsFEM for high-contrast problems}\label{sec:GMsFEM}
\subsection{Overview}\label{subsec:overview GMsFEM}
We recall the conforming GMsFEM construction for the high-contrast elliptic problem \cite{efendiev_generalized_2013,chung_residual-driven_2015,chung_adaptive_2016}. Let $D\subset\mathbb{R}^d$ be a bounded computational domain. We consider
\begin{equation}\label{eq:elliptic equation}
    \begin{split}
        \oprt{u(x)}&=f\quad\mr{in}\;D,\\
        u&=0\quad\mr{on}\;\partial D,
    \end{split}
\end{equation}
where $\ka(x)$ is a heterogeneous, high-contrast coefficient. Let $\mc{T}^H$ be a conforming coarse partition of $D$ with coarse elements $K$ and mesh size $H$. Its coarse nodes are denoted by $\{x_i\}_{i=1}^N$. For each coarse node $x_i$, the corresponding coarse neighborhood is
\begin{equation}\label{eq:coarse neighborhood}
    \omi=\bigcup\{K\in\mc{T}^H:x_i\in\overline{K}\}.
\end{equation}
Thus, $\omi$ is a union of coarse elements, rather than a set of nodes. Let $\mc{T}^h$ be a conforming refinement of $\mc{T}^H$, with $h\ll H$, and let $V$ be the conforming piecewise-linear finite element space on $\mc{T}^h$ satisfying the homogeneous boundary condition. In this work, $u\in V$ denotes the fine-grid reference solution, defined by
\begin{equation}\label{eq:CG variational form}
    a(u,v)=(f,v),\qquad \forall v\in V,
\end{equation}
where
\[
    a(u,v)=\int_D\ka(x)\nabla u\cdot\nabla v\;\mr{d}x,
    \qquad
    (f,v)=\int_Dfv\;\mr{d}x.
\]
We equip $V$ with the energy norm $\norm{v}_V^2=a(v,v)$. The fine mesh is assumed sufficiently resolved for $u$ to serve as the reference solution throughout the paper.

\begin{figure}[H]
    \centering
    \includegraphics[width=0.6\linewidth]{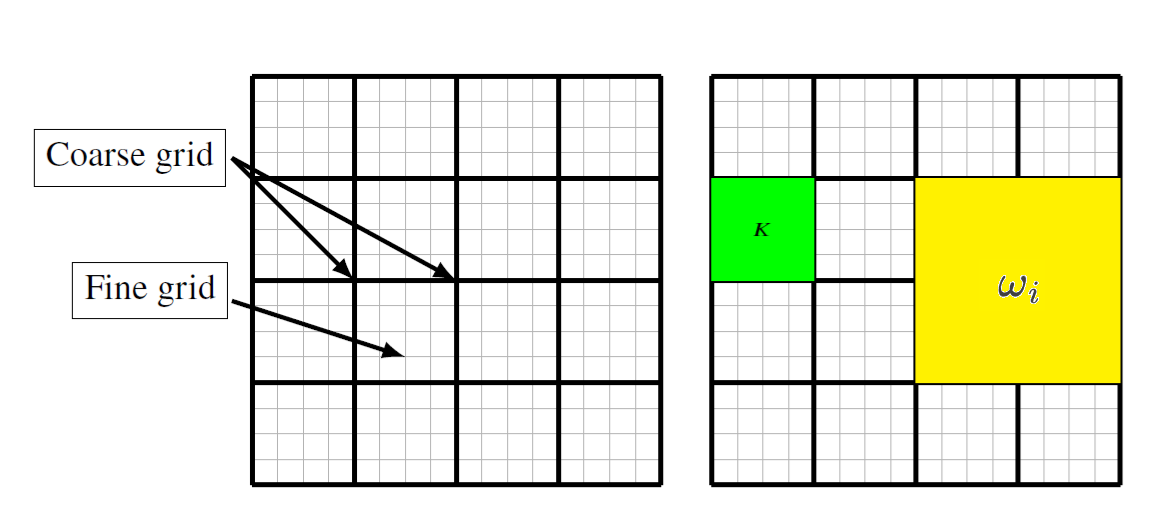}
    \caption{Illustration of a coarse neighborhood and coarse element.}
    \label{fig:CoarseFineMeshes}
\end{figure}

GMsFEM seeks a low-dimensional multiscale space while retaining the fine-scale effects of $\ka$. For each $x_i$, let $\{\psi_k^{\omi}\}_{k=1}^{l_i}$ be multiscale basis functions satisfying
\[
    \operatorname{supp}(\psi_k^{\omi})\subseteq\overline{\omi}.
\]
The number $l_i$ may vary with the coarse neighborhood. The multiscale space is
\[
    \vms=\operatorname{span}\{\psi_k^{\omi}:1\le i\le N,\ 1\le k\le l_i\},
\]
and the multiscale solution $\ums\in\vms$ solves
\begin{equation}\label{eq:variationalGMsFEM}
    a(\ums,v)=(f,v),\qquad \forall v\in\vms.
\end{equation}
Retaining multiple modes on a neighborhood allows the reduced space to represent fine-scale structures, such as distinct high-conductivity channels, that a single local mode may miss. Allowing $l_i$ to vary concentrates degrees of freedom where the local coefficient and solution are more complex and provides the mechanism used by adaptive enrichment.
The essential task is the construction of the local basis functions. We use offline basis functions, computed before the adaptive solution procedure, as the candidate functions for both the classical method and the proposed surrogate-assisted method.

\subsection{The offline basis functions}
For each coarse neighborhood $\omi$, the offline construction has three stages: form a rich snapshot space, select its dominant modes through a local spectral problem, and multiply the selected modes by a partition-of-unity function to obtain conforming offline basis functions. In this work, the snapshot space is formed by $\ka$-harmonic extensions. Let $\mc{J}_h(\omi)$ be the set of fine-grid nodes on $\partial\omi$, and let $L_i=|\mc{J}_h(\omi)|$. For each node $x_j\in\mc{J}_h(\omi)$, $j=1,\ldots,L_i$, define the discrete boundary delta function by
\begin{equation*}
    \delta_j^h(x_k)=\left\{\begin{array}{cc}
        1, & k=j, \\
        0, & k\neq j,
    \end{array}\right.
    \qquad x_k\in \mc{J}_h(\omi).
\end{equation*}
The corresponding snapshot function $\psi_j^{\omi,\snap}$ is the fine-grid solution of
\begin{equation*}
    \begin{aligned}
    -\nabla\cdot\left(\ka(x)\nabla\psi_j^{\omi,\snap}\right)&=0,\quad&&\mr{in}\;\omi,\\
    \psi_j^{\omi,\snap}&=\delta_j^h,\quad&&\mr{on}\;\partial\omi.
    \end{aligned}
\end{equation*}
Thus,
\[
    \vsnap^{\omi}=\operatorname{span}\{\psi_j^{\omi,\snap}:1\le j\le L_i\}.
\]

The offline space is obtained by selecting dominant modes of $\vsnap^{\omi}$ \cite{efendiev_multiscale_2011}. Define
\[
    a_{\omi}(v,w)=\int_{\omi}\ka(x)\nabla v\cdot\nabla w\;\mr{d}x,
    \qquad
    s_{\omi}(v,w)=\int_{\omi}\widetilde{\ka}(x)vw\;\mr{d}x,
\]
where
\[
    \widetilde{\ka}(x)
    =
    \ka(x)\sum_{r=1}^{N}H^2|\nabla\chi_r(x)|^2.
\]
Here, $\chi_i$ is the standard multiscale partition-of-unity function associated with $x_i$: for every $K\in\mc{T}^H$ satisfying $K\subset\omi$, it solves
\[
    -\nabla\cdot(\ka(x)\nabla\chi_i)=0\quad\mr{in}\;K,
    \qquad
    \chi_i=g_i\quad\mr{on}\;\partial K,
\]
where $g_i$ is the restriction to $\partial K$ of the continuous piecewise-linear coarse nodal function that equals one at $x_i$ and zero at the other vertices of $K$. The local spectral problem is
\begin{equation}\label{eq:original spectral problem}
    a_{\omi}(\phi,\varphi)=\lam s_{\omi}(\phi,\varphi),
    \qquad
    \forall\varphi\in\vsnap^{\omi}.
\end{equation}
Let $\{(\lam_k^{\omi},\Phi_k^{\omi})\}_{k=1}^{L_i}$ be the eigenpairs ordered so that
\[
    0<\lam_1^{\omi}\le\lam_2^{\omi}\le\cdots\le\lam_{L_i}^{\omi}.
\]
For an implementation, expand the snapshot functions in the nodal, piecewise-linear basis of $V$ restricted to $\omi$ and define
\[
    A_{mn}^{\omi}=a_{\omi}(\psi_m^{\omi,\snap},\psi_n^{\omi,\snap}),
    \qquad
    S_{mn}^{\omi}=s_{\omi}(\psi_m^{\omi,\snap},\psi_n^{\omi,\snap}),
    \qquad
    1\le m,n\le L_i.
\]
the spectral problem becomes
\begin{equation}\label{eq:discretized spectral problem}
    A^{\omi}\Phi_k^{\omi}=\lam_k^{\omi}S^{\omi}\Phi_k^{\omi},
\end{equation}
where $A^{\omi}=[A_{mn}^{\omi}]$ and $S^{\omi}=[S_{mn}^{\omi}]$. Retaining the $l_i$ eigenvectors associated with the smallest eigenvalues gives
\begin{equation*}
    \phi_k^{\omi,\mr{eig}}
    =
    \sum_{j=1}^{L_i}(\Phi_k^{\omi})_j\psi_j^{\omi,\snap},
    \qquad
    k=1,\ldots,l_i.
\end{equation*}
The conforming offline basis functions are
\begin{equation}\label{eq:local basis functions}
    \psi_k^{\omi}=\chi_i\phi_k^{\omi,\mr{eig}},
    \qquad
    k=1,\ldots,l_i.
\end{equation}
The local offline space is $\voff^{\omi}=\operatorname{span}\{\psi_k^{\omi}:1\le k\le l_i\}$, and the global offline space is
\[
    \voff
    =
    \operatorname{span}\{\psi_k^{\omi}:1\le i\le N,\ 1\le k\le l_i\}.
\]
When no adaptive enrichment has yet been performed, we take $\vms=\voff$ in (\ref{eq:variationalGMsFEM}).

\section{\hmns}\label{sec:hminus1 adaptive GMsFEM}
We next recall the offline adaptive GMsFEM \cite{chung_adaptive_2014,chung_adaptive_2016}. At enrichment level $m$, let $V_{\mathrm{off}}^m$ be the current multiscale space, and let $\ums^m\in V_{\mathrm{off}}^m$ solve
\[
    a(\ums^m,v)=(f,v),
    \qquad
    \forall v\in V_{\mathrm{off}}^m.
\]
Let $l_i^m$ be the number of offline basis functions currently retained on $\omi$. The local test space and its energy norm are
\[
    V_i=H_0^1(\omi)\cap V,
    \qquad
    \norm{v}_{V_i}^2=\int_{\omi}\ka(x)|\nabla v|^2\;\mr{d}x.
\]
The local residual is the linear functional
\begin{equation}\label{eq:Reisual}
    R_{m,i}(v)
    =
    \int_{\omi}fv\;\mr{d}x
    -
    \int_{\omi}\ka(x)\nabla\ums^m\cdot\nabla v\;\mr{d}x,
    \qquad
    \forall v\in V_i.
\end{equation}
Its dual norm
\[
    \norm{R_{m,i}}_{V_i^*}
    =
    \sup_{v\in V_i\setminus\{0\}}
    \frac{|R_{m,i}(v)|}{\norm{v}_{V_i}}
\]
is the weighted local $H^{-1}$ residual norm. The classical a posteriori estimate \cite{chung_adaptive_2014} gives
\begin{equation}\label{eq:error estimate of ums}
    \norm{u-\ums^m}_V^2
    \le
    \cerr\sum_{i=1}^{N}
    \norm{R_{m,i}}_{V_i^*}^2
    \left(\lam_{l_i^m+1}^{\omi}\right)^{-1},
\end{equation}
where $\cerr$ is the uniform constant from the cited estimate and $\lam_{l_i^m+1}^{\omi}$ is the first eigenvalue whose associated eigenfunction is not included in the local space on $\omi$. This motivates the local indicator
\begin{equation}\label{eq:hminus1 error indicator}
    \emind
    =
    \norm{R_{m,i}}_{V_i^*}^2
    \left(\lam_{l_i^m+1}^{\omi}\right)^{-1},
    \qquad
    i=1,\ldots,N.
\end{equation}
The indicator is nonnegative and measures the contribution of $\omi$ to the right-hand side of (\ref{eq:error estimate of ums}).

\begin{algorithm}[H]
\SetAlgoLined
\KwIn{Marking parameter $0<\theta<1$, initial space $V_{\mathrm{off}}^1$, enrichment size $s\ge1$, and maximum iteration count $M_{\max}$.}
\For{$m=1,2,\ldots,M_{\max}$}{
    \textbf{Step 1: Obtain the current multiscale solution}\\
    Find $\ums^m\in V_{\mathrm{off}}^m$ such that $a(\ums^m,v)=(f,v)$ for all $v\in V_{\mathrm{off}}^m$\;

    \textbf{Step 2: Compute error indicators}\\
    For each $i=1,\ldots,N$, compute $\emind$ from (\ref{eq:hminus1 error indicator})\;
    Choose a permutation $\sigma_m$ such that $\eta_{m,\sigma_m(1)}^2\ge\cdots\ge\eta_{m,\sigma_m(N)}^2$\;

    \textbf{Step 3: Marking}\\
    Find the smallest $k_m$ such that $\theta\sum_{i=1}^N\emind\le\sum_{j=1}^{k_m}\eta_{m,\sigma_m(j)}^2$\;

    \textbf{Step 4: Enrichment}\\
    \For{$j=1,\ldots,k_m$}{
        Set $i\leftarrow\sigma_m(j)$ and add $\{\psi_{l_i^m+1}^{\omi},\ldots,\psi_{l_i^m+s}^{\omi}\}$ to $V_{\mathrm{off}}^m$\;
        $l_i^{m+1} \leftarrow l_i^m + s$\;
    }
    For unmarked neighborhoods, set $l_i^{m+1}\leftarrow l_i^m$\;
    Construct $V_{\mathrm{off}}^{m+1}$ from the updated local offline spaces\;
}
\caption{\hmns}
\label{algo:hminus1}
\end{algorithm}

For a fixed discretization and local spectral construction, $\cerr$ is independent of the enrichment level; its precise uniformity is that of the a posteriori estimate in \cite{chung_adaptive_2014}. The integer $s$ is an implementation parameter specifying how many precomputed offline basis functions are added to each marked neighborhood, subject to the number of unused local modes. The algorithm is terminated after $M_{\max}$ iterations, or earlier when a prescribed error-indicator tolerance is reached. All enriched functions are selected from the precomputed offline basis. Hence, this procedure is an \emph{offline adaptive} GMsFEM; it should not be confused with residual-driven \emph{online-basis} construction, in which new basis functions are obtained by solving residual local problems \cite{chung_residual-driven_2015}.

The offline construction avoids recomputing candidate basis functions during the adaptive solution procedure. Nevertheless, the local dual norms in (\ref{eq:hminus1 error indicator}) must still be evaluated on every coarse neighborhood at every iteration. This repeated indicator evaluation is the computational component targeted by the GP surrogate introduced next.

\section{GP-based surrogate indicators}\label{sec:GP models}
The $H^{-1}$ residual indicator in (\ref{eq:hminus1 error indicator}) is an exact local quantity for a specified coefficient field and multiscale space. Its direct evaluation is the part of the offline-adaptive procedure that is replaced by a surrogate. The multiscale solve, the local spectral problems, and the enrichment itself are not replaced.

For a coarse neighborhood $\omi$ at enrichment level $m$, we use the feature vector
\begin{equation}\label{eq:vec unified mapping}
\mf{z}_{m,i}=\left(g_1(\omi,m),g_2(\omi,m),g_3(\omi,m),g_4(\omi,m),x_i^\top\right)^\top\in\rr^{d_f},
\end{equation}
where
\begin{equation*}
\begin{aligned}
g_1(\omi,m)&=\norm{\ums^m|_{\omi}}_{L^2(\omi)}, &
g_2(\omi,m)&=\norm{(\ums^m-\ums^{m-1})|_{\omi}}_{L^2(\omi)},\\
g_3(\omi,m)&=\lam_{l_i^m+1}^{\omi}, &
g_4(\omi,m)&=l_i^m.
\end{aligned}
\end{equation*}
Here $\ums^0$ denotes the fixed initialization state used to form the first feature vector. The coordinate $x_i\in\rr^d$ is retained as a vector, so that $d_f=4+d$; in the two-dimensional experiments reported below, the feature dimension is six rather than five. The first two features summarize the local solution magnitude and its change between consecutive enrichment levels, while $g_3$ and $g_4$ encode the next excluded local eigenvalue and the current local-space dimension. They are inexpensive proxies for the current local solution and spectral-truncation state, not replacements for the residual. Their ability to predict the indicator is therefore an empirical property of the specified training distribution. These compressed summaries do not contain the full local coefficient or residual state.

Let $Y_{m,i}=\eta_{m,i}^2$ denote the deterministic indicator obtained by the classical calculation for the state that generated $\mf{z}_{m,i}$. Since distinct local states can have the same compressed feature vector, we do not assume that $Y_{m,i}$ is an exact single-valued function of $\mf{z}_{m,i}$. Relative to the distribution of coefficient fields and adaptive states used for offline sampling, we write
\begin{equation}\label{eq:feature-discrepancy}
Y=m_\star(Z)+r,\qquad m_\star(z)=\mathbb{E}[Y\mid Z=z],\qquad \mathbb{E}[r\mid Z]=0.
\end{equation}
The term $r$ represents feature-compression or model discrepancy. It is not physical measurement noise: once a coefficient field, a discretization, and an adaptive state have been fixed, the computed label $Y$ is deterministic.

\begin{algorithm}[H]
\SetAlgoLined
\KwIn{Marking parameter $0<\theta<1$, initial offline space $\voff^1$, a pre-trained GP surrogate, enrichment parameter $s$, and maximum number of iterations $M_{\max}$.}
\For{$m=1,2,\ldots,M_{\max}$}{
    \textbf{Step 1: Multiscale solve}\\
    Find $\ums^m\in\voff^m$ such that $a(\ums^m,v)=(f,v)$ for all $v\in\voff^m$\;
    \textbf{Step 2: Surrogate scores}\\
    For every $i=1,\ldots,N$, construct $\mf{z}_{m,i}$ and obtain the posterior-mean score $\tilde{\eta}_{m,i}^2$ from Algorithm \ref{alg:GP_Inference}\;
    Find a permutation $\sigma_m$ such that $\tilde{\eta}_{m,\sigma_m(1)}^2\ge\cdots\ge\tilde{\eta}_{m,\sigma_m(N)}^2$\;
    \textbf{Step 3: Marking}\\
    Find the smallest $k_m$ such that $\theta\sum_{i=1}^N\tilde{\eta}_{m,i}^2\le\sum_{j=1}^{k_m}\tilde{\eta}_{m,\sigma_m(j)}^2$\;
    \textbf{Step 4: Enrichment}\\
    \For{$j=1,\ldots,k_m$}{
        Set $i\leftarrow\sigma_m(j)$ and add $\{\psi_{l_i^m+1}^{\omi},\ldots,\psi_{l_i^m+s}^{\omi}\}$ to the local offline space on $\omi$\;
        Set $l_i^{m+1}\leftarrow l_i^m+s$\;
    }
    For unmarked neighborhoods, set $l_i^{m+1}\leftarrow l_i^m$, and construct $\voff^{m+1}$ from the updated local spaces\;
}
\caption{GP-based offline-adaptive GMsFEM}
\label{alg:GP_GMsFEM}
\end{algorithm}

The superscript ``$2$'' in $\tilde{\eta}_{m,i}^2$ is retained to match the physical indicator convention. Since an unconstrained GP posterior mean for a positive target need not be positive, the marking argument in Section \ref{sec:analysis} explicitly assumes nonnegative surrogate scores. The present numerical pipeline uses posterior-mean predictions for ranking; enforcing positivity through a transformed-output or uncertainty-aware marking rule is a separate methodological extension.

\subsection{GP surrogate models}\label{subsec:GP surrogate models}
The training data are generated from $Q$ coefficient samples. For each sample, the classical indicator is evaluated on all $N$ coarse neighborhoods during $M_{\mathrm{tr}}$ enrichment levels. Thus,
\begin{equation}\label{eq:training set}
\mc{D}=\brcs{(\mf{z}_j,y_j)}_{j=1}^{J},\qquad J=Q M_{\mathrm{tr}}N,
\end{equation}
where $(\mf{z}_j,y_j)$ is one realization of $(Z,Y)$. The training distribution is fixed by this sampling design. Predictions outside the associated coefficient, mesh, source-term, and feature distribution are therefore extrapolations rather than a guaranteed extension of the method.

We place a GP prior on the regression function in (\ref{eq:feature-discrepancy}),
\begin{equation}\label{eq:prior GP}
m\sim\gp\left(\mu(z;\ze),c(z,z';\bbet,\sigma^2)\right),\qquad \mu(z;\ze)=\ze.
\end{equation}
We use the Mat\'ern $3/2$ and Mat\'ern $5/2$ covariance kernels to compare two finite-smoothness priors, with the former representing a rougher prior than the latter \cite{kanagawa_gaussian_2018}:
\begin{equation}\label{eq:Gaussian covariance functions}
\begin{aligned}
c_{3/2}(z,z';\bbet,\sigma^2)&=\sigma^2(1+\sqrt{3}\rho)\exp(-\sqrt{3}\rho),\\
c_{5/2}(z,z';\bbet,\sigma^2)&=\sigma^2\left(1+\sqrt{5}\rho+\frac{5}{3}\rho^2\right)\exp(-\sqrt{5}\rho),
\end{aligned}
\end{equation}
where $\rho=\sqrt{(z-z')^\top\operatorname{diag}(\bbet)(z-z')}$ and $\bbet\in\rr_+^{d_f}$ contains one inverse-length-scale parameter for each feature. The ARD parameterization permits different feature directions to have different fitted correlation scales; it does not by itself prove that an irrelevant feature has no effect.

Write $\mf{Z}=(\mf{z}_1,\ldots,\mf{z}_J)^\top$, $\mf{y}=(y_1,\ldots,y_J)^\top$, $\mf{C}=c(\mf{Z},\mf{Z};\bbet,\sigma^2)$, and $\bm{\mu}=\mu(\mf{Z};\ze)$. The fitted parameters $\Theta^*=(\ze^*,\bbet^*,\sigma^{*2},\delta^{*2})$ are obtained by minimizing the negative log marginal likelihood
\begin{equation}\label{eq:marginal log likelihood}
\Theta^*=\underset{\ze,\bbet,\sigma^2,\delta^2}{\arg\min}\left\{\frac12(\mf{y}-\bm{\mu})^\top(\mf{C}+\delta^2\mf{I}_J)^{-1}(\mf{y}-\bm{\mu})+\frac12\log\det(\mf{C}+\delta^2\mf{I}_J)+\frac{J}{2}\log(2\pi)\right\}.
\end{equation}
The learned nugget $\delta^2$ regularizes the covariance system and can absorb unresolved discrepancy $r$ in (\ref{eq:feature-discrepancy}); it is not interpreted here as independent physical noise corrupting the computed labels.

\subsection{Inference of error indicators}\label{subsec:inference of error indicators}
For a matrix $\hat{\mf{Z}}$ of online feature vectors, the posterior mean is
\begin{equation}\label{eq:posterior expectation}
\widehat{\mf{y}}=\mu(\hat{\mf{Z}};\ze^*)+c(\hat{\mf{Z}},\mf{Z};\bbet^*,\sigma^{*2})(\mf{C}+\delta^{*2}\mf{I}_J)^{-1}(\mf{y}-\bm{\mu}).
\end{equation}
Its $i$-th component is used as $\tilde{\eta}_{m,i}^2$ in Algorithm \ref{alg:GP_GMsFEM}. Once $\Theta^*$ is fitted, the operational predictor is the GP posterior mean, algebraically equivalent to a kernel ridge regression (KRR) estimator: with the convention $J^{-1}\sum_j(y_j-s(z_j))^2+\gamma\norm{s}_{\rkhs}^2$, the corresponding regularization parameter is $\gamma=\delta^{*2}/J$ \cite{rasmussen_gaussian_2005,kanagawa_gaussian_2018}. This equivalence does not require the deterministic simulator outputs to be independent Gaussian measurements.

Only the posterior mean is used by the present marking rule. The posterior variance is not used to change a mark, trigger an exact residual evaluation, or quantify a numerical error bar. Incorporating it in an uncertainty-aware marking strategy is left for future work.

\begin{algorithm}[H]
\SetAlgoLined
\KwIn{Training coefficient set $\{\ka_q\}_{q=1}^Q$, number of sampled enrichment levels $M_{\mathrm{tr}}$, and a Mat\'ern kernel $c$.}
\textbf{Phase 1: Offline data collection}\\
\For{$q=1,\ldots,Q$, $m=1,\ldots,M_{\mathrm{tr}}$, and $i=1,\ldots,N$}{
    Compute the exact label $y=\eta_{m,i}^2=\norm{R_{m,i}}_{V_i^*}^2(\lambda_{l_i^m+1}^{\omega_i})^{-1}$\;
    Construct $\mf{z}=\mf{z}_{m,i}$ from (\ref{eq:vec unified mapping}) and store $(\mf{z},y)$ in $\mc{D}$\;
}
\textbf{Phase 2: Offline model training}\\
Fit $\Theta^*$ using (\ref{eq:marginal log likelihood}) and precompute $\bm{\alpha}=(\mf{C}+\delta^{*2}\mf{I}_J)^{-1}(\mf{y}-\bm{\mu})$\;
\textbf{Phase 3: Online inference}\\
Construct the current feature matrix $\hat{\mf{Z}}$ and calculate $\widehat{\mf{y}}$ using (\ref{eq:posterior expectation})\;
\KwRet{Posterior-mean surrogate scores $\{\tilde{\eta}_{m,i}^2\}_{i=1}^N$.}
\caption{Construction and inference of GP-based surrogate indicators}
\label{alg:GP_Inference}
\end{algorithm}

\subsection{Cost and applicability}\label{subsec:cost-applicability}
The method is designed for a repeated-query regime in which one offline model is reused across sufficiently many admissible online solves. It transfers repeated online residual evaluation to an offline data-generation and training stage. Let $T_{\mathrm{data}}$ denote the cost of generating the $J=QM_{\mathrm{tr}}N$ labeled feature--indicator pairs in (\ref{eq:training set}), including the classical adaptive solves and local residual evaluations used for that purpose. With a dense covariance implementation, each factorization required during GP training costs $\mathcal{O}(J^3)$ operations and $\mathcal{O}(J^2)$ memory. Let $T_{\mathrm{train}}$ denote the total fitting cost, including the repeated factorizations used by hyperparameter optimization.

After fitting, the matrix expression in (\ref{eq:posterior expectation}) evaluates the posterior means for all $N$ neighborhoods in one batched or vectorized operation. With the precomputed weight vector in Algorithm \ref{alg:GP_Inference}, this requires $\mathcal{O}(NJd_f)$ kernel operations, apart from feature construction. Coefficient samples and many local residual evaluations in offline data generation can be parallelized; local feature construction and kernel evaluation likewise admit parallel work across neighborhoods. In contrast, the dense covariance factorization is a centralized bottleneck. Sparse or structured GP approximations may be needed when $J$ is substantially larger than in the present experiments \cite{quinonero-candela_unifying_2005,wu_variational_2022}.

For repeated online solves under the same training design, a symbolic break-even count is
\begin{equation}\label{eq:break-even}
n_{\mathrm{BE}}=\frac{T_{\mathrm{data}}+T_{\mathrm{train}}}{T_{H^{-1},\mathrm{online}}-T_{\mathrm{GP},\mathrm{online}}},
\end{equation}
provided that $T_{H^{-1},\mathrm{online}}>T_{\mathrm{GP},\mathrm{online}}$. The timing data reported in Section \ref{sec:Numerical experiments} measure the online indicator-evaluation component only; they do not provide all terms needed to estimate $n_{\mathrm{BE}}$. Thus, the existing results support the online component comparison, whereas an end-to-end break-even count requires the unreported offline terms. Reuse is intended for in-distribution queries under the same design; a new source term, mesh, local spectral construction, or coefficient distribution can change the feature--indicator relation and generally requires retraining.

\section{Analysis}\label{sec:analysis}
This section separates the established GMsFEM estimator from the additional approximation introduced by the surrogate. For any actual offline space $\voff^m$ generated by Algorithm \ref{alg:GP_GMsFEM}, the classical a posteriori estimate of Section \ref{sec:hminus1 adaptive GMsFEM} applies to the exact residual indicators of that space:
\begin{equation}\label{eq:hminus1 solution bound}
\norm{u-\ums^m}_V^2\le\cerr\sum_{i=1}^N\eta_{m,i}^2.
\end{equation}
In particular, the right-hand side contains $\eta_{m,i}^2$, not its GP prediction. The surrogate influences which neighborhoods are enriched; it does not alter the estimator itself.

\subsection{A nonuniform perturbed D\"orfler marking statement}\label{subsec:perturbed-marking}
Let $S_m=\{\sigma_m(1),\ldots,\sigma_m(k_m)\}$ be the set marked by Algorithm \ref{alg:GP_GMsFEM}. Pointwise score bounds may vary across neighborhoods, so the following result retains this nonuniform information.

\begin{theorem}[Nonuniform perturbed D\"orfler marking]\label{thm:perturbed-dorfler}
Assume that the predicted scores are nonnegative and that, for numbers $\varepsilon_{m,i}\ge0$,
\begin{equation}\label{eq:pointwise-surrogate-error}
\left|\tilde{\eta}_{m,i}^2-\eta_{m,i}^2\right|\le\varepsilon_{m,i},
\qquad i=1,\ldots,N.
\end{equation}
If $S_m$ satisfies the predicted D\"orfler condition
\begin{equation}\label{eq:predicted-dorfler}
\sum_{i\in S_m}\tilde{\eta}_{m,i}^2\ge\theta\sum_{i=1}^N\tilde{\eta}_{m,i}^2,
\end{equation}
then it satisfies the exact-indicator estimate
\begin{equation}\label{eq:perturbed-dorfler-nonuniform}
\sum_{i\in S_m}\eta_{m,i}^2
\ge
\theta\sum_{i=1}^N\eta_{m,i}^2
-\theta\sum_{i=1}^N\varepsilon_{m,i}
-\sum_{i\in S_m}\varepsilon_{m,i}.
\end{equation}
\end{theorem}
\begin{proof}
By (\ref{eq:pointwise-surrogate-error}) and (\ref{eq:predicted-dorfler}),
\begin{align*}
\sum_{i\in S_m}\eta_{m,i}^2
&\ge\sum_{i\in S_m}\tilde{\eta}_{m,i}^2-\sum_{i\in S_m}\varepsilon_{m,i}\\
&\ge\theta\sum_{i=1}^N\tilde{\eta}_{m,i}^2-\sum_{i\in S_m}\varepsilon_{m,i}\\
&\ge\theta\sum_{i=1}^N\eta_{m,i}^2
-\theta\sum_{i=1}^N\varepsilon_{m,i}
-\sum_{i\in S_m}\varepsilon_{m,i},
\end{align*}
which proves the assertion.
\end{proof}

If $\varepsilon_{m,i}\le\varepsilon_m$ for every $i$, then (\ref{eq:perturbed-dorfler-nonuniform}) recovers
\begin{equation}\label{eq:perturbed-dorfler}
\sum_{i\in S_m}\eta_{m,i}^2
\ge
\theta\sum_{i=1}^N\eta_{m,i}^2-(\theta N+|S_m|)\varepsilon_m.
\end{equation}
More generally, if $0<\rho<1$ and
\begin{equation}\label{eq:relative-surrogate-error}
\theta\sum_{i=1}^N\varepsilon_{m,i}+\sum_{i\in S_m}\varepsilon_{m,i}
\le
\rho\theta\sum_{i=1}^N\eta_{m,i}^2,
\end{equation}
then (\ref{eq:perturbed-dorfler-nonuniform}) yields
\begin{equation}\label{eq:reduced-dorfler}
\sum_{i\in S_m}\eta_{m,i}^2
\ge (1-\rho)\theta\sum_{i=1}^N\eta_{m,i}^2.
\end{equation}
Subject to the assumptions of the corresponding residual-adaptive GMsFEM convergence theorem, this gives the appropriate conditional connection to classical marking theory. The next subsection gives one conditional source of the pointwise quantities $\varepsilon_{m,i}$ in (\ref{eq:pointwise-surrogate-error}).

\subsection{Conditional KRR control of surrogate marking}\label{subsec:conditional-krr-bound}
The GP--KRR equivalence is an identity between the posterior-mean predictor and a KRR estimator; it does not by itself verify a prediction-error theorem. One possible route to the bounds in (\ref{eq:pointwise-surrogate-error}) is the deterministic KRR analysis of \cite{maddalena_deterministic_2021}. To apply that result without treating the computed labels as noisy physical measurements, we regard the realized discrepancy $r$ in (\ref{eq:feature-discrepancy}) as a deterministic approximation error.

\begin{assumption}[Conditional KRR setting]\label{assump:conditional-krr-setting}
After subtracting the fitted constant mean $\ze^*$, suppose that the following conditions hold. They are sufficient conditions for the argument below, not properties established by the current implementation or experiments.
\begin{enumerate}
    \item The feature domain $\mc{Z}\subset\rr^{d_f}$ is compact, and the fitted kernel $c$ is continuous and positive definite on $\mc{Z}\times\mc{Z}$.
    \item The training sites $\mf{z}_1,\ldots,\mf{z}_J$ are pairwise distinct, so the fitted kernel matrix $\mf{C}=c(\mf{Z},\mf{Z})$ is invertible.
    \item The fitted kernel, mean, and regularization parameters are held fixed, with $\delta^{*2}>0$ and $\gamma=\delta^{*2}/J$ under the KRR convention in Section \ref{subsec:inference of error indicators}.
    \item The centered target $f_\star:=m_\star-\ze^*$ belongs to $\rkhs$ and has an available, non-circular bound $\norm{f_\star}_{\rkhs}\le\Gamma$.
    \item With $\mf{y}_{\mathrm{c}}:=\mf{y}-\ze^*\mf{1}_J=f_\star(\mf{Z})+\mf{r}$, the realized training discrepancies satisfy $|\mf{r}|\le\bar{\mf{r}}$ for a known vector $\bar{\mf{r}}\in\rr_{>0}^J$, where both inequalities and absolute values are componentwise.
    \item For each online state under consideration, $|r_{m,i}|\le\bar r_{m,i}^{\mathrm{on}}$ for a specified $\bar r_{m,i}^{\mathrm{on}}\ge0$.
\end{enumerate}
\end{assumption}

For $z\in\mc{Z}$, set $\mf{c}_{\mf{Z}}(z):=c(\mf{Z},z)$. The fitted posterior-mean/KRR predictor is
\begin{equation}\label{eq:conditional-krr-predictor}
s^\star(z):=\ze^*+\mf{y}_{\mathrm{c}}^\top
(\mf{C}+\delta^{*2}\mf{I}_J)^{-1}\mf{c}_{\mf{Z}}(z).
\end{equation}
Define the kernel power function, the centered-data interpolant, and the bounded-discrepancy quantity by
\begin{align}
P_J(z)&:=\left(c(z,z)-\mf{c}_{\mf{Z}}(z)^\top\mf{C}^{-1}\mf{c}_{\mf{Z}}(z)\right)^{1/2},\label{eq:krr-power-function}\\
s_{\mathrm{int}}(z)&:=\mf{y}_{\mathrm{c}}^\top\mf{C}^{-1}\mf{c}_{\mf{Z}}(z),
\qquad
\norm{s_{\mathrm{int}}}_{\rkhs}^2=\mf{y}_{\mathrm{c}}^\top\mf{C}^{-1}\mf{y}_{\mathrm{c}},\label{eq:centered-interpolant}\\
\Delta_J&:=\max_{|\mf{q}|\le\bar{\mf{r}}}
\left(-\mf{q}^\top\mf{C}^{-1}\mf{q}
+2\mf{y}_{\mathrm{c}}^\top\mf{C}^{-1}\mf{q}\right).
\label{eq:krr-discrepancy-delta}
\end{align}
Also let
\begin{equation}\label{eq:krr-q-matrix}
\mf{Q}_J:=\mf{C}+(J\gamma)^{-1}\mf{C}^2
=\mf{C}+(\delta^{*2})^{-1}\mf{C}^2
\end{equation}
and define the fully specified pointwise quantity
\begin{align}
B_J(z):={}&P_J(z)
\left(\Gamma^2+\Delta_J-\mf{y}_{\mathrm{c}}^\top\mf{C}^{-1}\mf{y}_{\mathrm{c}}\right)^{1/2}
+\bar{\mf{r}}^\top\left|\mf{C}^{-1}\mf{c}_{\mf{Z}}(z)\right|\notag\\
&+\left|\mf{y}_{\mathrm{c}}^\top\mf{Q}_J^{-1}\mf{c}_{\mf{Z}}(z)\right|.
\label{eq:defined-krr-bound}
\end{align}
Since $s^\star-\ze^*$ is the centered KRR predictor with regularization $\gamma$, Theorem 1, equation (14), of \cite{maddalena_deterministic_2021}, applied to $f_\star$, gives
\begin{equation}\label{eq:conditional-krr-bound}
\left|s^\star(z)-m_\star(z)\right|\le B_J(z),
\end{equation}
where $s^\star$ is the posterior-mean/KRR predictor. The vector $\bar{\mf{r}}$ is an additional deterministic discrepancy bound used only in this conditional analysis. It is distinct from the learned nugget $\delta^{*2}$, which fixes the KRR regularization through $J\gamma=\delta^{*2}$.

\begin{proposition}[Conditional pointwise indicator-score bound]\label{prop:conditional-indicator-score-bound}
Suppose Assumption \ref{assump:conditional-krr-setting} holds. Then every online neighborhood under consideration satisfies
\begin{equation}\label{eq:krr-to-score-bound}
\left|\tilde{\eta}_{m,i}^2-\eta_{m,i}^2\right|
\le \varepsilon_{m,i}^{\mathrm{KRR}},
\qquad
\varepsilon_{m,i}^{\mathrm{KRR}}
:=B_J(\mf{z}_{m,i})+\bar r_{m,i}^{\mathrm{on}}.
\end{equation}
\end{proposition}
\begin{proof}
For the online state, $\eta_{m,i}^2=Y_{m,i}=m_\star(\mf{z}_{m,i})+r_{m,i}$ and $\tilde{\eta}_{m,i}^2=s^\star(\mf{z}_{m,i})$. Therefore,
\begin{align*}
\left|\tilde{\eta}_{m,i}^2-\eta_{m,i}^2\right|
&=\left|s^\star(\mf{z}_{m,i})-m_\star(\mf{z}_{m,i})-r_{m,i}\right|\\
&\le\left|s^\star(\mf{z}_{m,i})-m_\star(\mf{z}_{m,i})\right|+|r_{m,i}|\\
&\le B_J(\mf{z}_{m,i})+\bar r_{m,i}^{\mathrm{on}},
\end{align*}
where the last line uses (\ref{eq:conditional-krr-bound}) and Assumption \ref{assump:conditional-krr-setting}.
\end{proof}

\begin{corollary}[Conditional KRR-controlled D\"orfler marking]\label{cor:krr-controlled-dorfler}
Under Assumption \ref{assump:conditional-krr-setting}, suppose that the surrogate scores are nonnegative and that $S_m$ satisfies the predicted D\"orfler condition (\ref{eq:predicted-dorfler}). Then
\begin{equation}\label{eq:krr-controlled-dorfler}
\sum_{i\in S_m}\eta_{m,i}^2
\ge
\theta\sum_{i=1}^N\eta_{m,i}^2
-\theta\sum_{i=1}^N\varepsilon_{m,i}^{\mathrm{KRR}}
-\sum_{i\in S_m}\varepsilon_{m,i}^{\mathrm{KRR}}.
\end{equation}
If, in addition, there is a $0<\rho<1$ such that
\begin{equation}\label{eq:relative-krr-score-error}
\theta\sum_{i=1}^N\varepsilon_{m,i}^{\mathrm{KRR}}
+\sum_{i\in S_m}\varepsilon_{m,i}^{\mathrm{KRR}}
\le\rho\theta\sum_{i=1}^N\eta_{m,i}^2,
\end{equation}
then $S_m$ satisfies the exact D\"orfler condition (\ref{eq:reduced-dorfler}) with reduced parameter $(1-\rho)\theta$.
\end{corollary}
\begin{proof}
Proposition \ref{prop:conditional-indicator-score-bound} verifies the pointwise hypothesis of Theorem \ref{thm:perturbed-dorfler} with $\varepsilon_{m,i}=\varepsilon_{m,i}^{\mathrm{KRR}}$, so (\ref{eq:krr-controlled-dorfler}) follows from (\ref{eq:perturbed-dorfler-nonuniform}). Combining (\ref{eq:krr-controlled-dorfler}) with (\ref{eq:relative-krr-score-error}) gives (\ref{eq:reduced-dorfler}).
\end{proof}

\subsection{Scope and verification status}\label{subsec:surrogate-scope}
\begin{remark}[Verification status]\label{rem:krr-verification-status}
Theorem \ref{thm:perturbed-dorfler} is unconditional once the deterministic score bounds (\ref{eq:pointwise-surrogate-error}) hold. In contrast, Assumption \ref{assump:conditional-krr-setting} is not verified by the present experiments: they do not establish RKHS membership, a non-circular norm bound $\Gamma$, training or online discrepancy bounds, pairwise distinct or well-conditioned feature sites, or the applicability of a fixed-parameter analysis after data-dependent hyperparameter fitting. Proposition \ref{prop:conditional-indicator-score-bound} and Corollary \ref{cor:krr-controlled-dorfler} therefore identify a sufficient pathway and do not certify the current fitted surrogate.
\end{remark}

\begin{remark}[Feature-space norm and structured-sampling scope]\label{rem:norm-sampling-scope}
The Mat\'ern GP result reviewed in \cite{kanagawa_gaussian_2018} assumes, among other conditions, i.i.d. feature--label pairs, an input distribution on a normalized domain with density bounded away from zero and infinity, an additive Gaussian likelihood, explicit H\"older/Sobolev regularity of the regression function, and a specified relation between Mat\'ern smoothness and target regularity. Its regression error is measured in the feature-space norm $L^2(P_Z)$, not in the PDE energy norm $\norm{\cdot}_V$; related KRR rates additionally require a prescribed regularization schedule. The samples in (\ref{eq:training set}) are structured across coefficient realizations, enrichment levels, and coarse neighborhoods and have not been shown to satisfy the cited i.i.d. or density assumptions.
\end{remark}

Accordingly, we claim neither distribution-free generalization nor a universal asymptotic rate solely as a function of $J$. The required regularity and coverage may weaken when the coefficient field leaves the admissible distribution, when local eigenvalues cross or change ordering, or when the compressed features omit relevant local residual information. The numerical experiments provide in-distribution empirical evidence for the stated training and test design; they do not provide a KRR certification of the fitted model or a direct solution-error estimator based on predicted indicators.

\section{Numerical experiments}\label{sec:Numerical experiments}

In this section, we report a controlled held-out in-distribution validation of the surrogate-assisted marking mechanism against classical $\hmo$-residual indicators under the fixed coefficient-generation, source, and mesh design described below. The domain is $D=(0,1)^2$ and the source term is
\begin{equation*}
    f(x)=\exp\left(-\frac{\norm{x-\mr{cen}_1}_2^2}{2\times0.03^2}\right)+\exp\left(-\frac{\norm{x-\mr{cen}_2}_2^2}{2\times0.03^2}\right),
\end{equation*}
where $\mr{cen}_1=[0.15,0.15]^\top$ and $\mr{cen}_2=[0.85,0.85]^\top$. The norm is Euclidean. We use a $10\times10$ coarse grid, with each coarse block divided into $10\times10$ fine-grid blocks.

\begin{figure}[H]
    \centering
    \includegraphics[width=0.7\linewidth]{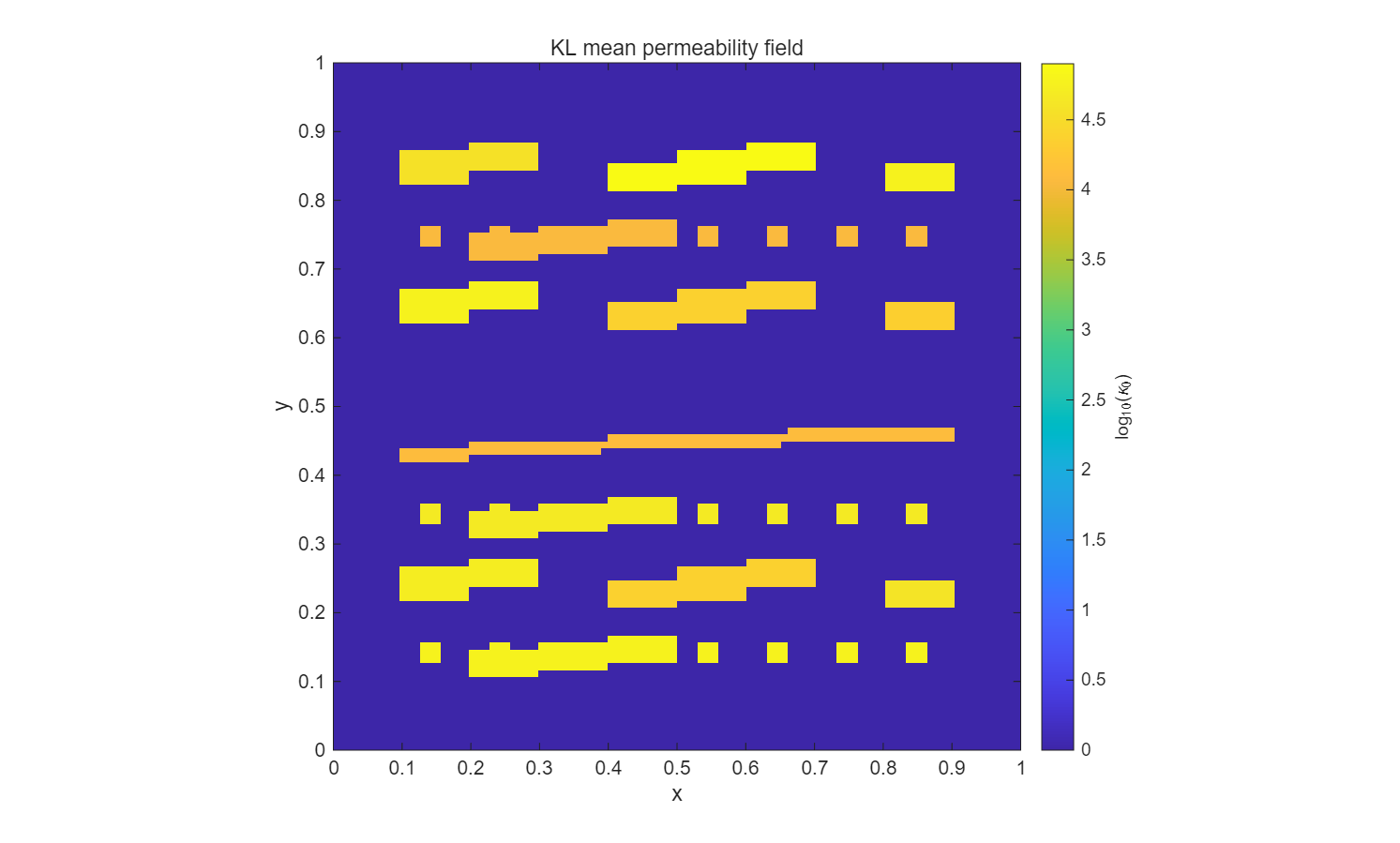}
    \caption{KL mean permeability field $\ka_0(x)$}
    \label{fig:Original Kappa Field}
\end{figure}

As is mentioned in Section \ref{subsec:GP surrogate models}, our aim is to use the trained GP model to predict the corresponding error indicators under different coefficients $\ka_q(x)$. To obtain various coefficient samples, the Karhunen-Lo\`eve (KL) expansion is applied to a given permeability field, as shown in Figure \ref{fig:Original Kappa Field}. This KL expansion method is briefly outlined here. For more details, please refer to \cite{wang_locally_2023}. Denote the permeability field shown in Figure \ref{fig:Original Kappa Field} by $\ka_0(x)$. Assume that the logarithmic permeability field $Y(x,\omega)=\log(\ka(x,\omega))$ is a second-order stationary Gaussian random field. To obtain the KL expansion of this random field, we define a covariance function
\begin{equation*}
    c_{\mr{KL}}(x,r)=\sigma_{\mr{KL}}^2\exp\left(-\frac{\norm{x-r}_2^2}{2\xi^2} \right),
\end{equation*}
where $x$ and $r$ are two arbitrary spatial variables in the domain $\Omega$. We present the eigenvalues and corresponding eigenfunctions of $c_{\mr{KL}}(x,r)$ as $(\lam_i,f_i),\;i=1,2,\ldots,+\infty$. The KL expansion of $Y(x,\omega)$ is written as
\begin{equation*}
    Y(x,\omega)\approx \mb{E}[Y(x,\omega)]+\sum_{i=1}^{N_k}\nu_i\sqrt{\lam_i}f_i,
\end{equation*}
where $\nu_i$ are standard identically independent Gaussian random variables. The permeability field $\ka_0(x)$ serves as the mean field $\mb{E}[Y(x,\omega)]$. The coefficient samples are transformed back by
\begin{equation*}
    \ka_q(x)=\exp(Y(x,\omega_q)).
\end{equation*}
In our experiment, we take $N_k=100$ and $\sigma_{\mr{KL}}=1$. For the spatial correlation length $\xi$, we use $\xi=\frac{1}{4},\frac{1}{8}$ to test the performance of the GP model, since the spatial variation and randomness of different coefficient samples are determined by $\xi$.

For each value of $\xi$, we generate 16 coefficient samples $\ka_q(x)$. The first 6 samples form the training set and the remaining 10 form the held-out test set; each sample is run for 20 enrichment iterations. This split defines the controlled in-distribution scope of the figures below: held-out coefficient realizations from the specified KL-generator family at each correlation length.

For the fitted GP models and held-out coefficient samples, Figure \ref{fig:full_figure} compares the mean relative energy-norm errors of classical $H^{-1}$-residual adaptive GMsFEM and the surrogate-guided method at each reported approximation dimension. The figure is intended to assess whether the surrogate-driven ranking yields an error-decay trend comparable with the classical indicator in this fixed design.

Within the four configurations shown in Figure \ref{fig:full_figure}, the Mat\'ern $3/2$ curves give the more favorable error-decay behavior relative to the Mat\'ern $5/2$ curves. One plausible explanation is that the lower prior smoothness of Mat\'ern $3/2$ better accommodates localized variation in the feature--indicator relation. This is an interpretation compatible with the reported curves, not a causal identification or a universal kernel preference. For both $\xi=1/4$ and $\xi=1/8$, the surrogate tracks the classical error-decay trend on the tested held-out samples; robustness beyond these settings is not evaluated here.

Table \ref{tab:Speedup} reports the average online indicator-evaluation speedup for the four experimental configurations. The reported ratio is
\begin{equation*}
    S=\frac{T_{H^{-1}}}{T_{\mr{GP}}},
\end{equation*}
where $T_{H^{-1}}$ is the measured time for direct classical $H^{-1}$-indicator evaluation and $T_{\mr{GP}}$ is the measured time for feature construction and GP posterior-mean prediction. These timings exclude offline data generation and GP hyperparameter training; the full amortized cost is discussed in Section \ref{subsec:cost-applicability}.

For these held-out samples, the reported online ratios range from 2.0 to 2.1. The averages are taken over the test coefficient samples for each configuration. They show that, in the reported setting, posterior-mean scores can be evaluated faster than direct local dual norms while retaining the error-decay comparison shown in Figure \ref{fig:full_figure}. The full offline and amortized cost boundary is stated once in Section \ref{subsec:cost-applicability}; the table reports the online component only.

\begin{figure}[H]
    \centering
    \begin{subfigure}{0.48\textwidth}
        \centering
        \includegraphics[width=\linewidth]{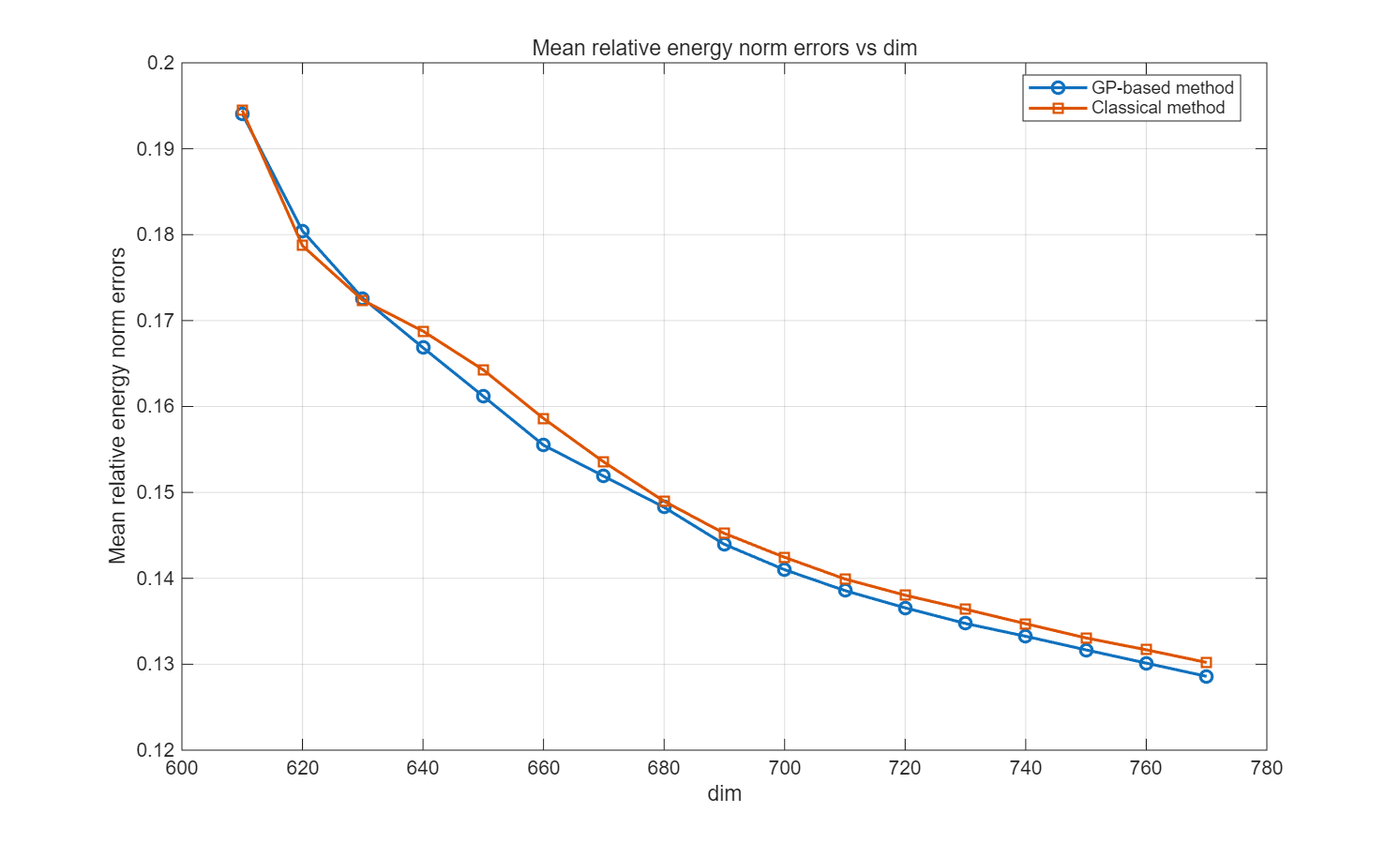}
        \caption{$\xi=\frac{1}{4}$,\;Mat\'ern 3/2}
        \label{subfig:1/4,3/2}
    \end{subfigure}
    \hfill 
    \begin{subfigure}{0.48\textwidth}
        \centering
        \includegraphics[width=\linewidth]{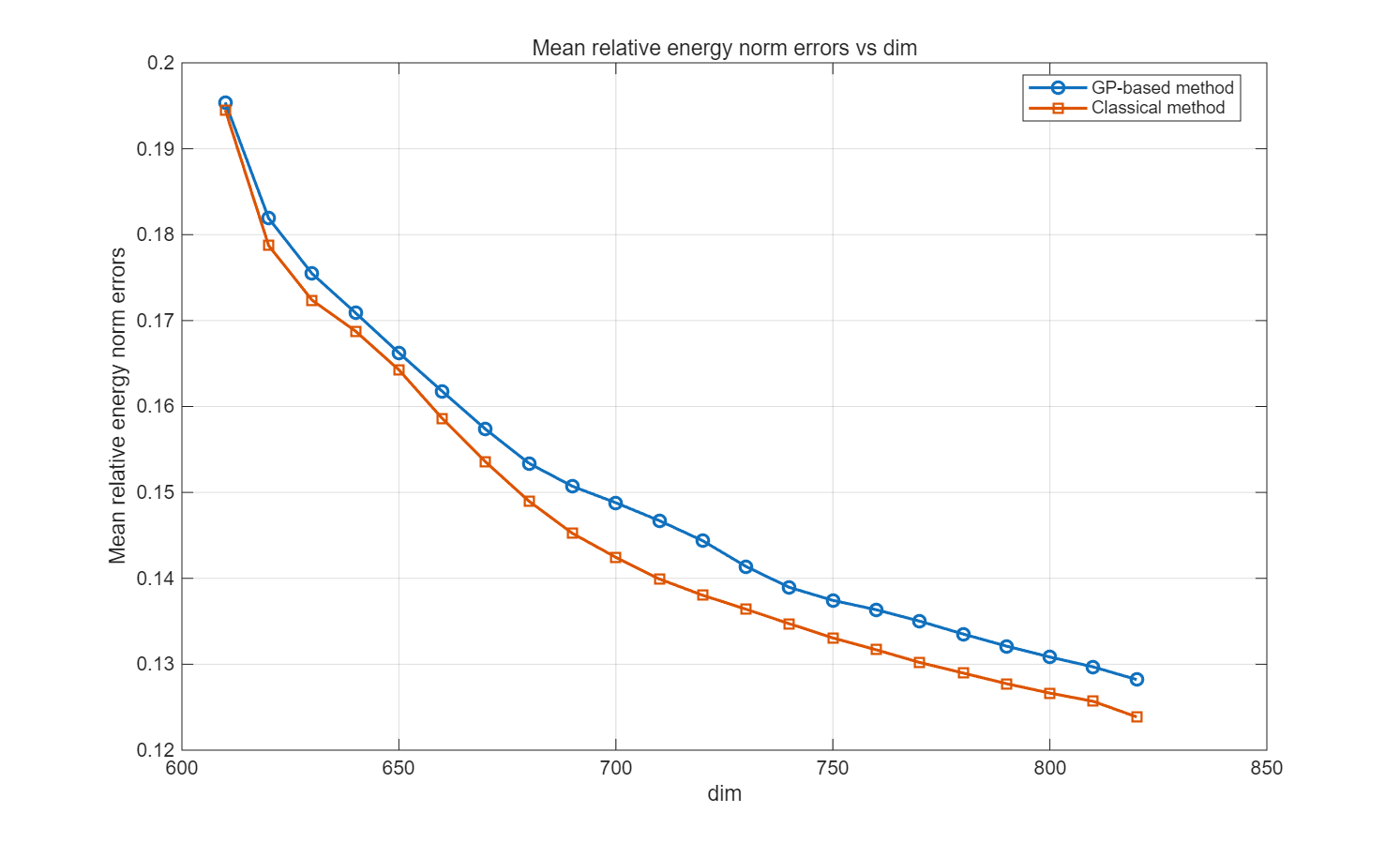}
        \caption{$\xi=\frac{1}{4}$,\;Mat\'ern 5/2}
        \label{subfig:1/4,5/2}
    \end{subfigure}

    \vspace{2ex} 

    \begin{subfigure}{0.48\textwidth}
        \centering
        \includegraphics[width=\linewidth]{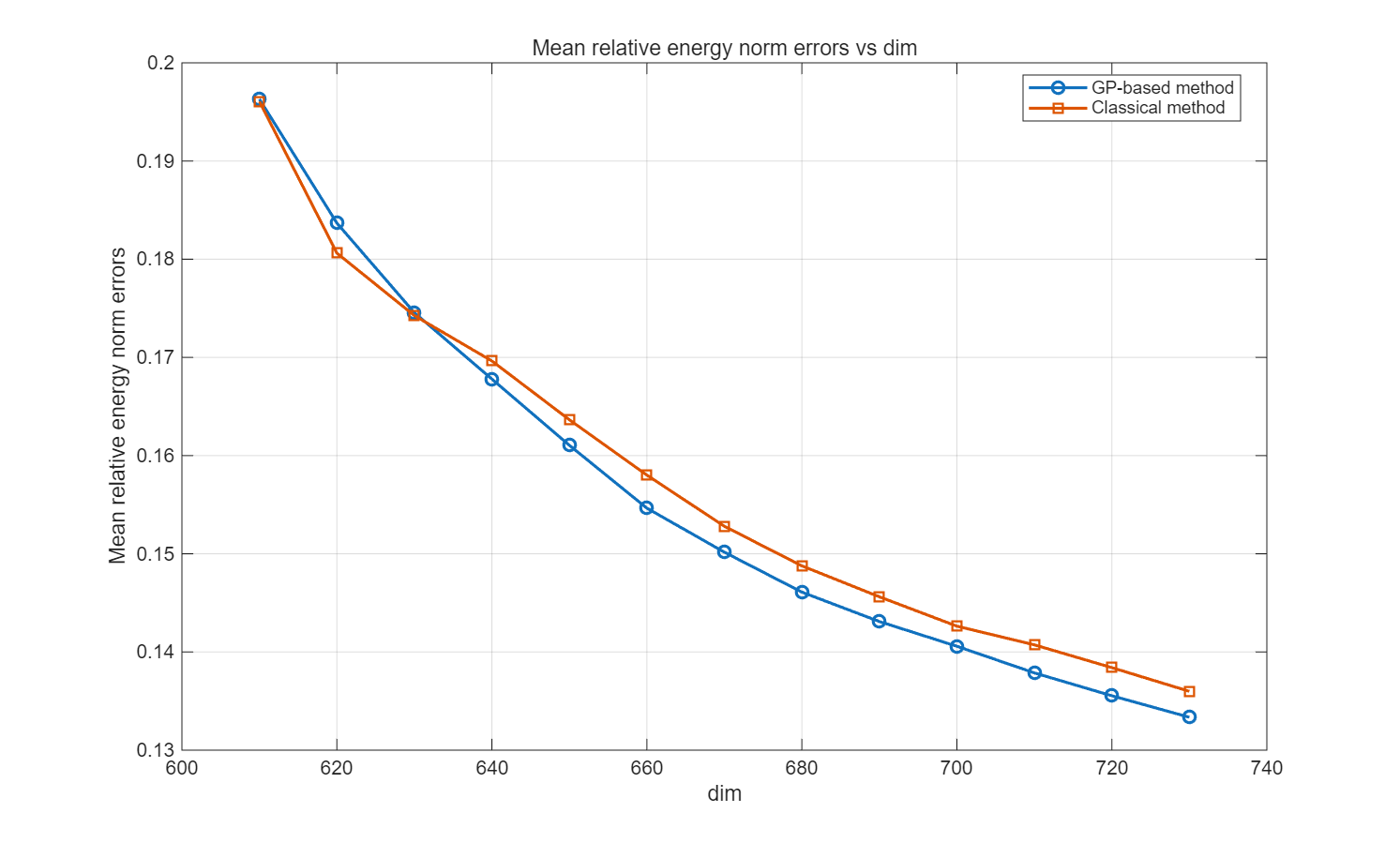}
        \caption{$\xi=\frac{1}{8}$,\;Mat\'ern 3/2}
        \label{subfig:1/8,3/2}
    \end{subfigure}
    \hfill
    \begin{subfigure}{0.48\textwidth}
        \centering
        \includegraphics[width=\linewidth]{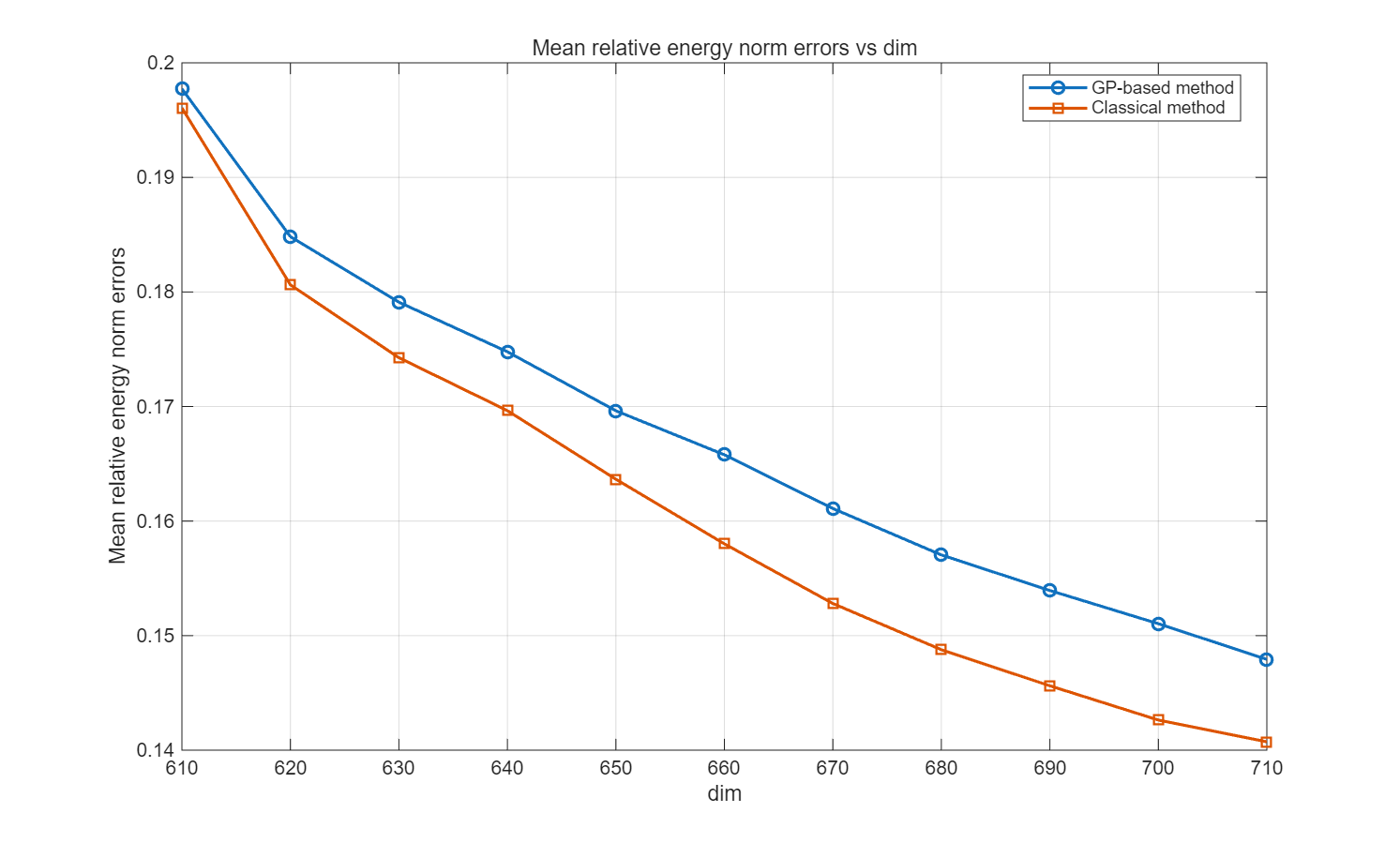}
        \caption{$\xi=\frac{1}{8}$,\;Mat\'ern 5/2}
        \label{subfig:1/8,5/2}
    \end{subfigure}

    \caption{Mean relative energy-norm errors on held-out coefficient samples}
    \label{fig:full_figure}
\end{figure}

\begin{table}[H]
    \centering
    \begin{tabular}{c|c|c|c|c}
        \toprule
         & $\xi=\frac{1}{4}$,Mat\'ern 3/2 & $\xi=\frac{1}{4}$,Mat\'ern 5/2 & $\xi=\frac{1}{8}$,Mat\'ern 3/2 & $\xi=\frac{1}{8}$,Mat\'ern 5/2 \\
        \midrule 
       Speedup  & 2.1 & 2.0 & 2.1 & 2.1 \\
       \bottomrule
    \end{tabular}
    \caption{Average online indicator-evaluation speedup on held-out samples (offline data generation and GP training excluded)}
    \label{tab:Speedup}
\end{table}

\section{Conclusions}

We introduced a non-intrusive GP posterior-mean surrogate for the local $H^{-1}$-residual indicator scores used to mark neighborhoods in offline-adaptive GMsFEM. The surrogate targets repeated indicator evaluation and ranking, while the multiscale solve, local spectral construction, and basis enrichment remain standard GMsFEM components. Compact local solution and spectral features permit batched neighborhood inference, and the fitted model can be reused across repeated admissible queries.

On the controlled held-out coefficient samples, surrogate-guided enrichment gives error-versus-DoF trends comparable with the classical residual-indicator procedure. The Mat\'ern $3/2$ curves are more favorable than the Mat\'ern $5/2$ curves in the four reported configurations, and the measured online indicator-evaluation speedup is $2.0$--$2.1$, excluding offline data generation and GP training.

The analysis retains the classical residual estimator for the exact indicators of the space produced by surrogate-guided enrichment. Under explicit pointwise score-accuracy bounds, the nonuniform perturbed D\"orfler result shows how the predicted marking rule captures exact residual indicator mass. A bounded-discrepancy KRR estimate supplies one conditional route to these score bounds under the additional assumptions stated in Section \ref{subsec:conditional-krr-bound}.

The present evidence is confined to the reported coefficient generator, mesh, source, and train--test settings, and the timing result is component-level rather than end-to-end. Applying the surrogate to a materially different distribution or discretization requires reassessing training coverage and generally retraining the model. Future work includes computable score-bound diagnostics, positivity-preserving and uncertainty-aware marking, broader distributional tests, and scalable GP training for larger offline data sets.

\section*{Acknowledgement}

The research of Eric Chung is partially supported by the Hong Kong RGC General Research Fund (Projects: 14305423 and 14305624).

\bibliographystyle{unsrt}
\bibliography{GPReference} 

\end{document}